\documentclass[12pt,twoside]{amsart}
\usepackage[dvips]{graphics}
\usepackage{a4,amsmath,amssymb,amsfonts,amscd,mathrsfs}
\reversemarginpar
\newtheorem{theorem}[subsection]{Theorem}
\newtheorem{lemma}[subsection]{Lemma}

\newtheorem{proposition}[subsection]{Proposition}
\newcommand{\Harm}{{\rm Harm}}

\newcommand{\Mat}{{\rm Mat}}
\newcommand{\Orth}{{\rm O}}
\newcommand{\SL}{{\rm SL}}

\newcommand{\dee}{\partial}
\newcommand{\harm}{{\rm harm}}
\newcommand{\id}{{\rm id}}
\newcommand{\inv}{^{-1}}
\renewcommand{\phi}{\varphi}
\newcommand{\rk}{{\rm rk}}
\newcommand{\sca}[2]{\left\langle #1, #2 \right\rangle}

\newcommand{\edop}{{\mathbb E}}
\newcommand{\ndop}{{\mathbb N}}
\newcommand{\qdop}{{\mathbb Q}}
\newcommand{\rdop}{{\mathbb R}}
\newcommand{\zdop}{{\mathbb Z}}
\newcommand{\Bcal}{{\mathcal B}}
\author{Juan Marcos Cervi\~no}
\address{Fakult\"at f\"ur Mathematik, Universit\"at Duisburg-Essen,
45117 Essen, Germany}
\email{juan.cervino@uni-due.de}
\author{Georg Hein}
\address{Fakult\"at f\"ur Mathematik, Universit\"at Duisburg-Essen,
45117 Essen, Germany}
\email{georg.hein@uni-due.de}
\date{August 31, 2009}
\begin{document}
\title{lattice invariants from the heat kernel (II)}
\begin{abstract}
Given an integral lattice $\Lambda$ of rank $n$
and a finite sequence $m_1 \leq m_2
\leq \dots \leq m_k$ of natural numbers we construct a modular form 
$\Theta_{m_1,m_2,\dots,m_k,\Lambda}$ of level $N=N(\Lambda)$.
The weight of this modular form is $nk/2+\sum_{i=1}^k m_k$.
This construction generalizes the theta series $\Theta_\Lambda$ of
integral lattices, because $\Theta_\Lambda = \Theta_{0,\Lambda}$.\\
We give the $q$-expansions  of the modular forms $\Theta_{m,m,\Lambda}$,
and $\Theta_{1,1,1,\Lambda}$ and show that (up to some scaling) they are
given by power series with integer coefficients.\\
\end{abstract}

\subjclass[2000]{11F11, 11F27, 11E45}
\keywords{spherical theta functions, lattice invariants, modular form}
\maketitle
\section{Introduction}
For an integral lattice $\Lambda$ the theta series $\Theta_\Lambda(\tau)
= \sum_{\lambda \in \Lambda} \exp(2 \pi i \|\lambda\|^2 \tau)$ is a
first invariant. Moreover, $\Theta_\Lambda$ is a modular form of weight
$\rk(\Lambda)/2$ and level $N(\Lambda)$. Since the vector space of
modular forms of given weight and level is of finite dimension,
$\Theta_\Lambda$ can be read off from the first coefficients in its
$q$-expansion. Unfortunately, there are pairs $(\Lambda,\Lambda')$ of
lattices which possess the same theta series $\Theta_\Lambda =
\Theta_{\Lambda'}$ and are not isometric. A first example are the two
unimodular lattices $E_8 \oplus E_8$ and $E_{16}$ of rank 16 (see page
1243 in \cite{Elk2} for details). Schiemann constructed in \cite{Sch} an
example of two four dimensional lattices $(\Lambda,\Lambda')$ which are
isospectral (i.e. $\Theta_\Lambda = \Theta_{\Lambda'}$) but not
isometric.\\
Spherical theta functions $\Theta_{h,\Lambda} :=  \sum_{\lambda \in
\Lambda} h(\lambda)\exp(2 \pi i \|\lambda\|^2 \tau)$ define for
homogeneous harmonic polynomials $h$ also modular forms of level
$N(\Lambda)$. These modular forms depend on $h$ and are of weight
$\deg(h) +\rk(\Lambda)/2$. (The term {\em spherical theta functions}
appears in \cite{Zag} whereas Elkies uses {\em weighted theta function}
in \cite{Elk}.) The authors managed in \cite{CH} to find sums of
products of these spherical theta functions which give new lattice
invariants.  These modular forms can be used to distinguish the two
isospectral lattices in Schiemann's example (see Proposition 4.4 in
\cite{CH}).  In our article \cite{CH} we construct an invariant
$c_{m_1,m_2,\dots,m_k,\Lambda}$ which turns out to be a sum of products
of modular forms and their derivatives.  Out of these invariants we can
sometimes construct invariant harmonic data
$p_{m_1,m_2,\dots,m_k,\Lambda}$ which give invariant modular forms.
However, we computed only $p_{1,1,\Lambda}$ explicitly for lattices
$\Lambda$ of arbitrary rank in our article \cite{CH}.\\
The aim of this article is a direct construction of the invariant
harmonic datum $p_{m_1,m_2,\dots,m_k,\Lambda}$. So take an integral
lattice $\Lambda$ of rank $n$, let $N$ be the level of $\Lambda$, and
fix a finite sequence $0 \leq m_1 \leq m_2 \leq \dots \leq m_k$ of
integers.  We start with an isometric embedding $\Lambda \to \edop^n$ of
$\Lambda$ into the Euclidean space. Here we consider $\Lambda$ as a
distribution on the Schwartz functions on $\edop^n$. The heat flux of
this distribution is given by a function $f_\Lambda: \rdop^+ \times
\edop^n \to \rdop$. Using the {\em harmonic Taylor coefficients} of
$f_\Lambda$ we obtain the harmonic invariant system
$p_{m_1,m_2,\dots,m_k,\Lambda}$ which provides a modular form
$\Theta_{m_1,m_2,\dots,m_k,\Lambda}$ of level $N$ and weight
$nk/2+\sum_{i=1}^km_k$, independent from the chosen embedding
$\Lambda \to \edop^n$ (see Theorem \ref{thetam1mk}).\\
Next we give for all integers $m \geq 0$ the $q$-expansion of the
invariant modular forms $\Theta_{m,m,\Lambda} = \sum_{k \geq 0}
a_{m,m,k}q^k$. It turns out that the coefficients $a_{m,m,k}$ are given
by
\[ a_{m,m,k} = \sum_{(v,w) \in
\Lambda^2, \|v\|^2+\|w\|^2=k} p_m(\cos(\measuredangle
(v,w ))) \|v\|^{2m}\|w\|^{2m} \, ,\]
where $p_m$ is an even polynomial of degree $2m$ (see Theorem
\ref{main1}). We compute these polynomials in Lemma \ref{pm}. The first
ones being
\[ p_0(c) =1  \,,\,\,\,
p_1(c) = \frac{c^2}{2}-\frac{1}{2n} \, , \mbox{ and }
p_2(c) = \frac{c^4}{24}-\frac{c^2}{4(n+4)}+\frac{1}{8(n+4)(n+2))} \, .
\]
Knowing these polynomials we can give the modular forms $\Theta_{m,m,E_8}$
for the $E_8$ lattice for $m \leq 9$ in \ref{ex-e8}.
We conclude with computing the {\em triple modular invariant}
$\Theta_{1,1,1,\Lambda}$ in Theorem \ref{theta-111}.

\subsection*{Acknowledgment}
This work has been supported by the SFB/TR 45
``Periods, moduli spaces and arithmetic of algebraic varieties''.

\section{The modular forms $\Theta_{k_1,\dots,k_m,\Lambda}$}\label{THETAmm}
\subsection{Notation}
We consider a lattice $\Lambda \subset \edop^n$ embedded in the
$n$-dimensional euclidean space.
In \cite{CH} we defined the function $f_\Lambda:\rdop^+ \times \edop^n
\to \rdop$ by
\[ f_\Lambda(t,x) = {(4\pi t)^\frac{-n}{2}}\sum_{\gamma \in \Lambda}
\exp \left( \frac{-\|x-\gamma\|^2}{4t} \right) \, . \]
As explained in Section 2.1 of \cite{CH} this function describes the
heat flux of the lattice $\Lambda$.
We call a function $c_{\Lambda}:\rdop^+ \to \rdop$, which we obtain from
$f_\Lambda$, a lattice invariant if the action of the orthogonal group
$\Orth(n)$ on the isometric embeddings of $\Lambda \to \edop^n$ does not
change $c_{\Lambda}$. The first example for such a lattice invariant is
$c_{0,\Lambda}(t)=f_\Lambda(t,0)$.
It was shown in \cite[Section 2.10]{CH} that this function
$c_{0,\Lambda}$ determines the theta series of the lattice $\Lambda$.
We call a lattice $\Lambda$ integral if the square lengths $\|\gamma
\|^2$ are integers for all $\gamma \in \Lambda$.
An integral lattice $\Lambda$ has two integer invariants: its
discriminant $D = D(\Lambda)$, and the level $N=N(\Lambda)$.
We recall their definitions. If $\frac{1}{2}A \in \Mat_{n \times
n}(\rdop)$ is a symmetric Gram matrix for $\Lambda$, then $A \in \Mat_{n
\times n}(\zdop)$ and has even integers on its diagonal.
We set $D(L):= \det(A)$.
The smallest positive integer $N$ such
that $N\cdot A\inv \in \Mat_{n \times n}(\zdop)$ and $N\cdot A\inv $ has
even entries on the principal diagonal is called the level of $\Lambda$.

\subsection{Harmonic polynomials}\label{harm-num}
We list some well known properties of harmonic polynomials. For proofs see
\cite[Theorem 3.1]{Hel}, and
\cite[Chapter XIII, Exercises 33--35]{Lang}.
The ring of polynomial functions on
$\edop^n$ we denote by $A=\rdop[x_1,\dots,x_n]$.
The ring $A$ has a
natural grading $A= \oplus_{k \geq 0} A_k$ with $A_k$ the homogeneous
polynomials of degree $k$.
We define a pairing on
\[ A \times A \to \rdop \, ,\quad \langle g , f \rangle =
g\left(\frac{\dee}{\dee x_1},\dots \frac{\dee}{\dee x_n}\right)f |_0
\,.\]
The orthogonal group $\Orth(n)$ acts on $\edop^n$, defining an action on
the polynomials $\sigma(f)(x):=f(\sigma\inv(x))$
It is convenient, to recall some well known basic properties of this
pairing.
\begin{enumerate}
\item The pairing is a bilinear, symmetric, and positive definite.
\item For $I=(i_1,\dots,i_n) \in \ndop^n$ we set $x^I = \prod_{k=1}^n
x_k^{i_k}$, and $I!=\prod_{k=1}^n i_k!$. The normed monomials $\left\{
\frac{x^I}{\sqrt{I!}} \right\}_{I \in \ndop^n}$ form a orthonormal basis.
\item For two polynomials $P,Q \in A$ we have $\langle P \cdot Q ,f
\rangle= \langle P , Q\left(\frac{\dee}{\dee x_1},\dots \frac{\dee}{\dee
x_n}\right)f \rangle$. In particular, the Laplace operator $\Delta=
-\sum_{k=1}^n \frac{\dee^2}{\dee x_k^2}:A_m \to A_{m-2}$ has up to sign
the multiplication with $r^2=\sum_{k=1}^n x_k^2$ as adjoint.
\item The infinite dimensional representation $\Orth(n) \times A \to A$
is compatible with the grading of $A$ and with the inner product on
$A$. Therefore it defines finite
dimensional representations $\Orth(n) \to \Orth(A_m, \langle, \rangle )$
for all $m \geq 0$.
\item The kernel $\Harm_m$ of $\Delta:A_m \to A_{m-2}$ is an irreducible
representation. The vector space $\Harm_m$ of harmonic functions
is of dimension $\binom{n+m-1}{n-1}  -  \binom{n+m-3}{n-1}$. Its
orthogonal complement is $r^2 \cdot A_{m-2}$ which is an $\Orth(n)$
invariant subspace. In consequence, we obtain the decomposition of $A_m$
into irreducible subspaces
\[A_m = \bigoplus_{k=0}^{\lfloor m/2 \rfloor} r^{2k} \Harm_{m-2k}
\,.\]
\item For $h_1,h_2 \in \Harm_{2m-2k}$ we have the equality
\[ \langle r^{2k}h_1 , r^{2k}h_2 \rangle = a_{k,m}\sca{h_1}{h_2}
\mbox{ with } a_{k,m}=2^k k!
\prod_{l=1}^{k} (n+4m-2k-2l) \]
\item For $h \in \Harm_{m}$, and natural numbers $k,d \in \ndop$ we have
\[ r^{2k}\Delta^k(r^{2d}h) = b_{k,d,m} r^{2d}h \mbox{ with } b_{k,d,m} =
\prod_{l=0}^{k-1} (2l-2d)(n-2+2d-2l+2m)
\, . \]
Note that $b_{0,d,m}=1$, and $b_{k,d,m} = 0 \iff k >d$.
\item Using the above numbers $b_{k,d,m}$ we define for $m \geq 0$, and
$k=0, \dots, \lfloor m/2 \rfloor$ rational numbers $d_{k,m}$ by
the assignment $r_{0,m}:= 1$, and
\[r_{d,m} := \frac{-1}{b_{d,d,m-2d} }
\sum_{k=0}^{d-1} r_{k,m}b_{k,d,m-2d} \mbox{ for all integers }
d=1,\dots , \lfloor m/2 \rfloor  \, . \]
The use of the integers $b_{k,d,m}$ is not necessary. Using their
definition we obtain:
\[  r_{0,m} = 1 \mbox{ and by  }  r_{d,m} = - \sum_{k=0}^{d-1} \left(
\prod_{l=k}^{d-1} \frac{1}{(2l-2d)(n-2+2m-2d-2l)} \right) r_{k,m} \, .\]
From this, we derive in Lemma \ref{combi-1} the explicit formula:
\[ r_{k,m}\inv=2^k k! \prod_{l=0}^{k-1} (n+2m-4-2l) .\]
However, the first definition implies immediately that for all
$d \geq 1$ we have:
\[ \sum_{k=0}^{d} r_{k,m}b_{k,d,m-2d} =0 \, . \]
Therefore we conclude, that the linear map $P_{\harm,m} = \sum_{k \geq
0}^{m/2}r_{k,m} r^{2k} \Delta^k$ operates on $A_m$ as the harmonic
projection. Indeed, for a homogeneous harmonic function $h$ of degree
$m-2d$ we find that $P_{\harm,m}(r^{2d}h) = \left( \sum\limits_{k=0}^{d}
r_{k,m}b_{k,d,m-2d}\right) r^{2d}h$.
\item Let us explicitly give the harmonic projections in degree two, four
and six:
\[\begin{array}{rcl}
P_{\harm,2}& =&\id+\frac{1}{2n}r^2\Delta \\
P_{\harm,4} & =&
\id+\frac{1}{2(n+4)}r^2\Delta+\frac{1}{8(n+2)(n+4)}r^4\Delta^2\\
P_{\harm,6} & =&
\id+\frac{1}{2(n+8)}r^2\Delta+\frac{1}{8(n+6)(n+8)}r^4\Delta^2+
\frac{1}{48(n+4)(n+6)(n+8)}r^6\Delta^3\\
\end{array}
\]
\end{enumerate}

\subsection{Harmonic Taylor coefficients}
We consider the homogeneous parts of the Taylor expansion of $f_\Lambda$
at the point $x=0$. Since $f_\Lambda$ is symmetric in $x$, only the even
parts appear. We set
\[ f_{\Lambda,m} = \sum _{I \subset \ndop^n, |I|=2m} a_I \frac{x^I}{I!}
\mbox{
where for } I=(i_1,i_2,\dots,i_n) \,\,\,
I!:=\prod_{m=1}^n i_m!  \mbox{ ,  } x^I:=\prod_{m=1}^n x_m^{i_m}
\,\mbox{ , and }\]
\[ a_I:= \langle x^I , f_\Lambda \rangle = \frac{\dee^{i_1}}{\dee
x_1^{i_1}}
\frac{\dee^{i_2}}{\dee x_2^{i_2}} \cdots \frac{\dee^{i_n}}{\dee
x_n^{i_n}}f_\Lambda|_{\rdop^{+} \times \{0\}} \, .\]

The $f_{\Lambda,m}$ are homogeneous polynomials of degree $2m$ in the
$x_i$ over the ring of functions in $t$.
We may write
\[ f_{\Lambda,m} = \sum _{I \subset \ndop^n, |I|=2m}
\sca{\frac{x^I}{\sqrt{I!}}}{f_\Lambda}
\frac{x^I}{\sqrt{I!}} \,.\]
Indeed, the operator $f \mapsto \sum _{I \subset \ndop^n, |I|=2m}
\sca{\frac{x^I}{\sqrt{I!}}}{f_\Lambda}
\frac{x^I}{\sqrt{I!}}$ is the identity on $A_{2m}$ the space of
homogeneous polynomials of degree $2m$. So we can replace the
orthonormal basis $\left\{ \frac{x^I}{\sqrt{I!}} \right\}_{I \subset
\ndop^n, |I|=2m}$ by any other orthonormal basis.
If $\Bcal^\harm_{2m}$ is an orthonormal basis of $\Harm_{2m}$, the space of
harmonic polynomials of degree $2m$, then the projection of
$f_{\Lambda,m}$ to $\Harm_{2m}$ is called the {\em harmonic Taylor
coefficient} of $f_\Lambda$ and given by
\[f^\harm_{\Lambda,m} = \sum_{h \in \Bcal_{2m}^\harm} \sca{h}{f_\Lambda}h \,
.\]
We derive more formulas for $f^\harm_{\Lambda,m}$ which we will use in
the sequel. Taking any orthonormal basis $\Bcal_{2m}$ of $A_{2m}$ we
obtain $f^\harm_{\Lambda,m} = \sum_{g \in \Bcal_{2m}}
\sca{g}{f_\Lambda}P_\harm(g)$ where $P_\harm:A_{2m} \to A_{2m}$
denotes the orthogonal projection to the space of harmonic polynomials.
Therefore, we conclude $f^\harm_{\Lambda,m} = \sum_{g \in \Bcal_{2m}}
\sca{P^*_\harm(g)}{f_\Lambda}g$ with $P^*_\harm$ the adjoint operator of
$P_\harm$. Since an orthogonal projection is self adjoint we find that
$f^\harm_{\Lambda,m} = \sum_{g \in \Bcal_{2m}}
\sca{P_\harm(g)}{f_\Lambda}g$. Using the formula for the harmonic
projection developed in \ref{harm-num}.(8) we derive the next
\begin{proposition}\label{harm-tay}
Let $\Bcal$ be any orthonormal basis of $A_{2m}$. We have:
\[f^\harm_{\Lambda,m} = \sum_{g \in \Bcal_{2m}}
\sca{P_\harm(g)}{f_\Lambda}g = \sum_{g \in \Bcal_{2m}}
\sca{g}{P_\harm(f_\Lambda)}g= \qquad \qquad \qquad \qquad \qquad \qquad\]
\[\qquad \qquad = \sum_{g \in \Bcal_{2m}}
\sca{g}{\sum_{k=0}^m p_{k,2m} r^{2k}\Delta^k f_\Lambda}g = 
\sum_{g \in \Bcal_{2m}}
\sca{g}{\sum_{k=0}^m p_{k,2m} r^{2k}
(-1)^k\frac{\dee^k}{\dee t^k}f_\Lambda}g
\]
with the rational numbers $p_{k,2m}=\frac{1}{2^kk!} \prod_{l=0}^{k-1}
(n+4m-4-2l)\inv$ from
\ref{harm-num}.(8).
\end{proposition}
\begin{proof}
We have shown all equalities but the last one. This is a
consequence of the identity
$\Delta^kf_\Lambda = (-1)^k\frac{\dee^k}{\dee t^k}f_\Lambda$
(see 2.1 in \cite{CH}).
\end{proof}

\subsection{The invariant harmonic system $p_{m_1,\dots,m_k,\Lambda}$}
We define for any set of integers $m_1,\dots,m_k$ with $m_i \geq 0$ the
function
$p_{m_1,\dots,m_k,\Lambda}$ by
\[ p_{m_1,\dots,m_k,\Lambda} := \int_{S^{n-1}} f_{\Lambda,m_1}^\harm
\cdot f_{\Lambda,m_2}^\harm  \cdot \dots \cdot f_{\Lambda,m_k}^\harm d
\bar \mu \, . \]
If $\phi: \edop^n \to \edop^n$ is any isometry, then $\phi$ commutes
with the multiplication with $r^2$ as well as with $\Delta$. Whence it
commutes with the harmonic projection, which can be described in terms
of $r^2$ and $\Delta$. In consequence $p_{m_1,\dots,m_k,\phi(\Lambda) }
= p_{m_1,\dots,m_k,\Lambda}$.
Using the defining equation of $f_{\Lambda,m_i}$
we can write
\[ p_{m_1,\dots,m_k,\Lambda} = \sum_{h_1 \in \Bcal_{2m_1}^\harm \dots
h_k \in \Bcal_{2m_k}^\harm} \left( \prod_{l=1}^k \sca{h_l}{f_\Lambda}
\right) \int_{S^{n-1}} h_1 \cdot \ldots \cdot h_k d \bar \mu \, , \]
where $\Bcal^\harm_{2m_i}$ is an orthonormal basis of harmonic
polynomials of degree $2m_i$.
Since the $\int_{S^{n-1}} h_1 \cdot \ldots \cdot h_k d \bar \mu $ are
merely real numbers we obtain from this equation and the equality
$p_{m_1,\dots,m_k,\phi(\Lambda) } = p_{m_1,\dots,m_k,\Lambda}$ for all
isometries $\phi \in \Orth(n)$, that $p_{m_1,\dots,m_k,\Lambda}$ is an
invariant harmonic system for lattices $\Lambda \in \edop^n$ (see
\cite[2.8]{CH}). The Proposition 2.9 from \cite{CH} implies:

\begin{theorem}\label{thetam1mk}
For any integral lattice $\Lambda \subset \edop^n$ the modular form
\[ \Theta_{m_1,\dots,m_k,\Lambda} = \sum_{h_1 \in \Bcal_{2m_1}^\harm
\dots h_k \in \Bcal_{2m_k}^\harm}
\left( \prod_{l=1}^k \Theta_{h_l,\Lambda} \right)
 \int_{S^{n-1}} h_1 \cdot \ldots \cdot h_k d \bar \mu \]
is a modular form of weight $\frac{nk}{2}+2\sum_{l=1}^km_i$. 
The modular form is of level $N$, the level of the lattice $\Lambda$.
Furthermore, $\Theta_{m_1,\dots,m_k,\Lambda}$ is independent from the
chosen embedding $\Lambda \to \edop^n$. If $k$ is an odd number, then
$\Theta_{m_1,\dots,m_k,\Lambda}$ has character $\left(\frac{D}{\cdot}
\right)$. For $k$ an even integer $\Theta_{m_1,\dots,m_k,\Lambda}$ is a
modular form for the trivial character. \qed
\end{theorem}

The functions $\left\{ \frac{x^I}{\sqrt{I!}} \right\}_{I \subset
\ndop^n, |I|=2m}$ form an orthonormal basis of $A_{2m}$. Whereas an
orthonormal basis of the subspace $\Harm_{2m} \subset A_{2m}$ is more
difficult. However by Proposition \ref{harm-tay} we can compute
$p_{m_1,\dots,m_k,\Lambda}$ in a different manner:
\[  p_{m_1,\dots,m_k,\Lambda} = \sum_{g_1 \in \Bcal_{2m_1} \dots
g_k \in \Bcal_{2m_k}} \left( \prod_{l=1}^k \sca{P_\harm(g_l)}{f_\Lambda}
\right) \int_{S^{n-1}} g_1 \cdot \ldots \cdot g_k d \bar \mu \, , \]
with the $\Bcal_{2m_i}$ orthonormal basis of $A_{2m_i}$.
Applying \cite[Proposition 2.9]{CH} to this presentation of
$p_{m_1,\dots,m_k,\Lambda}$ we obtain the next

\begin{proposition}\label{alt-def}
The modular form $\Theta_{m_1,\dots,m_k,\Lambda}$ can be computed using
orthonormal basis $\Bcal_{2m_i}$ of $A_{2m_i}$, and the orthogonal
harmonic projections $P_\harm:A_{2m_i} \to A_{2m_i}$ as follows
\[ \Theta_{m_1,\dots,m_k,\Lambda} = \sum_{g_1 \in \Bcal_{2m_1}
\dots g_k \in \Bcal_{2m_k}}
\left( \prod_{l=1}^k \Theta_{P_\harm(g_l),\Lambda} \right)
\int_{S^{n-1}} g_1 \cdot \ldots \cdot g_k d \bar \mu \, . \]
\end{proposition}

\section{The modular forms $\Theta_{m,m,\Lambda}$ for integers $m
\geq 0$}
\subsection{Definition of $\Theta_{m,m,\Lambda}$}
On the real polynomials on $\edop^n$ we have two $\Orth(n)$-invariant scalar
products. The one defined in \ref{harm-num}, and the integral scalar
product $\sca{f}{g}_2:= \int_{S^{n-1}} fg d \bar \mu$. The first has the
advantage that $A_m \bot A_k$ for $m \ne k$. However, when we restrict
to the irreducible subspace $\Harm_{2m}$ the two scalar products agree
up to a constant $c_{2m}$, i.e. $\sca{f}{g}_2= c_{2m}\sca{f}{g}$
for all $f,g \in
\Harm_{2m}$.
The formula from Theorem \ref{thetam1mk} yields
\[ \begin{array}{rcl}
\Theta_{m,m,\Lambda}
& =& \displaystyle
\sum_{h_1 \in \Bcal_{2m}^\harm} \sum_{h_2 \in \Bcal_{2m}^\harm}
\Theta_{h_1,\Lambda} \Theta_{h_2,\Lambda} \int_{S^{n-1}} h_1h_2 d \bar
\mu \\
& =& \displaystyle
\sum_{h_1 \in \Bcal_{2m}^\harm} \sum_{h_2 \in \Bcal_{2m}^\harm}
\Theta_{h_1,\Lambda} \Theta_{h_2,\Lambda}
\sca{h_1}{h_2}_2\\
& =& \displaystyle
c_{2m} \sum_{h \in \Bcal_{2m}^\harm} 
\Theta_{h,\Lambda}^2 \, .
\end{array}\]
A straightforward computation using \cite[Corollary A.2]{CH} yields
$c_{2m}=\prod\limits_{k=0}^{2m-1}\frac{1}{n+2k}$.
However, to ease notation we define $\Theta_{m,m,\Lambda}:= \sum_{h \in
\Bcal_{2m}^\harm} \Theta_{h,\Lambda}^2$.

\begin{lemma}\label{aux-32}
Let $v$ and $w$ to vectors in $\edop^n$, $\Bcal_{2m}^\harm$ be a
orthonormal basis of $\Harm_{2m}$, and $c=\cos(\measuredangle
(v,w )) = \frac{\sca{v}{w}}{\|v\|\|w\|}$ be the cosine of the
angle between $v$ and $w$. We have
\[ \sum_{h \in \Bcal_{2m}^\harm} h(v)h(w) = p_m(c)\|v\|^{2m}\|w\|^{2m}
\]
where $p_m$ is the even polynomial of degree $2m$ from Lemma \ref{pm}.
\end{lemma}
\begin{proof}
We proceed by induction on $m$. For $m=0$ the statement is obvious.
Let $\Bcal_{2m}$ be a orthonormal basis of $A_{2m}$. We want to compute
$D_m(v,w) = \sum_{h \in \Bcal_{2m}} h(v)h(w)$. Since $\Orth(n)$ acts
orthogonal on $A_{2m}$, this number is independent of the chosen basis.
In particular, we can take $\Bcal_{2m} = \left\{ \frac{x^I}{\sqrt{I!}}
\right\}_{|I| = 2m }$. With this choice we find $D_m(v,w) =
\frac{\sca{v}{w}^{2m}}{(2m!)}$. Another choice for an orthonormal basis
is by \ref{harm-num}.(6) 
\[\Bcal_{2m}'= B_{2m}^\harm  \cup
\frac{1}{\sqrt{a_{1,m}}} r^2 B_{2m-2}^\harm \cup \dots \cup
\frac{1}{\sqrt{a_{m,m}}}
r^{2m}B^\harm_0 \, , \]
where the numbers $a_i,m$ are those defined in \ref{harm-num}.(6).
This basis corresponds to the irreducible
decomposition $A_{2m} = \bigoplus_{k=0}^m r^{2k} \Harm_{2m-2k}$.
Working with this  orthonormal basis we deduce
\[D_m(v,w) = \sum_{k=0}^m \frac{1}{a_{k,m}} \sum_{h \in B^\harm_{2m-2k}
} h(v)h(w) \, . \]
The defining equation for the polynomials $p_m$, Lemma \ref{pm}, and the
induction hypothesis yield the stated formula for $p_m$.
\end{proof}

\begin{theorem}\label{main1}
For an integer lattice $\Lambda \subset \edop^n$ the modular form
\[ \Theta_{m,m,\Lambda} = \sum_{h \in \Bcal_{2m}^\harm}
\Theta_{h,\Lambda}^2 \]
is of weight $4m+n$, has level $N(\Lambda)$, and is independent of the
chosen embedding. Its $q$-expansion is given by
\[ \Theta_{m,m,\Lambda} (\tau) = \sum_{k \geq 0} \left( \sum_{(v,w) \in
\Lambda^2, \|v\|^2+\|w\|^2=k} p_m(\cos(\measuredangle
(v,w ))) \|v\|^{2m}\|w\|^{2m} \right) q^k \, . \]
For $m >0$ we have,
$ \frac{(2m)!2^{2m}\prod_{l=0}^{m-1} (n+4m-4-2l) }{q^{2l}}
\Theta_{m,m,\Lambda} \in \zdop[[q]]$
where $l$ denotes the minimum of $\|v\|^2$ for all nonzero $v \in
\Lambda$.
\end{theorem}
\begin{proof}
We take the function $\Theta_{m,m,\Lambda} = \sum_{h \in
\Bcal_{2m}^\harm} \Theta_{h,\Lambda}^2 $. As a sum of squares of modular
forms of weight $2m+\frac{n}{2}$ it is a modular form  of weight$4m+n$.
Let us calculate the $q$-expansion:
\[ \Theta_{m,m,\Lambda}(\tau) = \sum_{h \in \Bcal_{2m}^\harm} 
\sum_{(v,w) \in \Lambda \times \Lambda } h(v)h(w) q^{\|v\|^2+\|w\|^2}
\quad \mbox{ with } q=\exp(2 \pi i \tau)
\,. \]
This yields by Lemma \ref{aux-32} the stated $q$-expansion. From the
$q$-expansion we directly deduce that $\Theta_{m,m,\Lambda}$ is
independent from the chosen embedding $\Lambda \to \edop^n$. Likewise we
see that the coefficients of $q^k$ in  $\Theta_{m,m,\Lambda}$ vanish for
all $k < 2l$. Now we consider for two vectors $v,w \in \Lambda$ the
number
\[ \delta_m(v,w) = (2m)!2^{2m}\prod_{l=0}^{m-1} (n+4m-4-2l)
p_m(\cos(\measuredangle (v,w ))) \|v\|^{2m}\|w\|^{2m} \, . \]
This number is by Lemma \ref{pm} and the definition of the cosine
given by
\[ \delta_m(v,w) = \sum_{k=0}^m
(-1)^k\frac{(2m)!}{(2m-2k)!k!} 2^{2m-k}\sca{v}{w}^{2m-2k}
\|v\|^{2k}\|w\|^{2k} \prod_{l=k}^{m-1}(n+4m-4-2l) \, .
\]
Since $\Lambda$ is integral $\sca{v}{w} \in \frac{1}{2}\zdop$. Thus,
$\delta_m(v,w)$ is a sum of integers. This completes the proof.
\end{proof}

\subsection{Example: The modular forms $\Theta_{m,m,E_8}$}\label{ex-e8}
Let $v \in E_8$ be any lattice vector of length one.
Basic combinatorics yield, that the possible values of $\sca{v}{w}^2$
for the 240 lattice vectors $w \in E_8$ of length one are:
one (2 times), $\frac{1}{4}$ (112 times), and $0$ (126 times).
This allows by Theorem \ref{main1} the computation of the coefficient
$a_{m,m,2}$ of $q^2$ in the $q$-expansion of $\Theta_{m,m,E_8}$ for all
integers $m \geq 1$. We list the first:
\[ \begin{array}{c|c}
m & a_{m,m,2} \\
\hline
m \in \{1,2,3,5 \} & 0 \\
4 &  3/896\\
6 & 7/316293120\\
7 & 1/30057431040\\
8 & 1/22235892940800 \\
9 & 1/21727643959296000 \\
\end{array}\]

Now $\Theta_{m,m,E_8}$ is a cusp form of weight $4m+8$ for
$\SL_2(\zdop)$ which starts with $\Theta_{m,m,E_8} = a_{m,m,2}q^2 +
\dots $. Since we know the dimensions of the spaces of cusp forms we can
determine $\Theta_{m,m,E_8}$ from $a_{m,m,2}$ for $m \leq 6$, and
obtain:
\[ \begin{array}{rcl}
\Theta_{m,m,E_8}(\tau) & = & 0 \quad \mbox{for  } m \in \{1,2,3,5 \} \\
\\
\Theta_{4,4,E_8}(\tau) & = & \displaystyle\frac{3}{896} \Delta^2(\tau)\\
\\
\Theta_{6,6,E_8}(\tau) & = &
\displaystyle\frac{7}{658944} G_8(\tau)\Delta^2(\tau) \,
,\\
\end{array} \]
where $\Delta(\tau)=q\prod_{n \geq 1}(1-q^n)^{24}$ is the discriminant
function, and $G_8(\tau) = \frac{1}{480}+\sum_{n \geq 1} \sigma_7(n)q^n$
is the Eisenstein series of weight eight (see Chapter 0 in \cite{Zag}).
The vanishing of $\Theta_{m,m,E_8}(\tau)$ for $m \in \{1,2,3,5 \}$ can
be deduced alternatively: These modular forms are sums of squares of
cusp forms of weight $2m+4$ which do not exists for those values of $m$.
The same argument can be applied for $m \in \{7,8,9\}$. Here
$\Theta_{m,m,E_8}$ is a sum of squares of cusp forms of weight $2m+4$.
Those cusp forms form a one-dimensional vector space with generator, say
$G_{2m-8}\Delta$, where $G_{2m-8}$ is the Eisenstein series of weight
$2m-8$. We deduce that $\Theta_{m,m,E_8}$ is a scalar multiple of
$G_{2m-8}^2\Delta^2$. We
can use the coefficient $a_{m,m,2}$ to deduce this scalar. After all,
this yields:
\[ \begin{array}{rclcl}
\Theta_{7,7,E_8}(\tau) &
= & \displaystyle \frac{9}{1064960} G_6(\tau)^2\Delta(\tau)^2
&& \mbox{with } G_6(\tau) = \frac{-1}{504}+\sum_{n \geq 1}
\sigma_5(n)q^n\, ,\\\\
\Theta_{8,8,E_8}(\tau) & = & \displaystyle\frac{1}{96509952}
G_8(\tau)^2\Delta(\tau)^2\, , \\\\
\Theta_{9,9,E_8}(\tau) &
= & \displaystyle\frac{11}{3429236736000} G_{10}(\tau)^2\Delta(\tau)^2 
&& \mbox{with } G_{10}(\tau) = \frac{-1}{264}+\sum_{n \geq 1}
\sigma_9(n)q^n \, .\\

\end{array}\]

\section{The modular form $\Theta_{1,1,1,\Lambda}$}
\subsection{The invariant harmonic datum
$p_{1,1,1,\Lambda}$}\label{aux-41}
We consider the harmonic Taylor coefficient $f^\harm_{\Lambda,1}$ of the
function $f_\Lambda$. By Proposition \ref{harm-tay} it is given by
\[f^\harm_{\Lambda,1} =
\sum_{i =1}^n \sca{x_i^2-\frac{1}{n}\sum_{j=1}^nx_j^2}{f_\Lambda}
\frac{x_i^2}{2} + \sum_{1 \leq i < j \leq n}
\sca{x_ix_j}{f_\Lambda}x_ix_j\, . \]
When introducing the shorthand
$h_i=\sca{nx_i^2-\sum_{j=1}^nx_j^2}{f_\Lambda}$, and 
$b_{ij}=\sca{x_ix_j}{f_\Lambda}$ we obtain
\[2n f^\harm_{\Lambda,1} = \sum_{i =1}^n h_i x_i^2 + 2n\sum_{1 \leq i < j
\leq n} b_{ij}x_ix_j \, . \]
We consider the invariant
harmonic datum:
\[ p_{1,1,1,\Lambda} =
\int_{S^{n-1}} \left(f^\harm_{\Lambda,1}\right)^3 d\bar \mu \,.\]
We need the following spherical integrals
$ \int_{S^{n-1}} x_i^6d\bar \mu = \frac{15}{n(n+2)(n+4)}$, 
 $\int_{S^{n-1}} x_i^4x_j^2 d\bar \mu = \frac{3}{n(n+2)(n+4)}$, and
$\int_{S^{n-1}} x_i^2x_j^2x_k^2 d\bar \mu = \frac{1}{n(n+2)(n+4)}$.
Furthermore $\int_{S^{n-1}} \prod_i x_i^{n_i} d \bar \mu =0$ when at
least one of the exponents $n_i$ is an odd integer (see \cite[Corollary
A.2]{CH}). After these preparation we compute:
\[ \begin{array}{rcl}
n(n+2)(n+4)(2n)^3 p_{1,1,1,\Lambda} & = &n(n+2)(n+4) \int_{S^{n-1}} 2n
(f^\harm_{\Lambda,1})^3 d \bar \mu \qquad \qquad \qquad \\\end{array}\]
\[ \begin{array}{rcl}
&=& n(n+2)(n+4) \int_{S^{n-1}} \left( \sum_{i =1}^n h_i x_i^2 +
2n\sum_{1 \leq i < j \leq n} b_{ij}x_ix_j \right)^3 d \bar \mu \\
&=&\displaystyle \sum_{i_1=1}^n\sum_{i_2=1}^n
\sum_{i_3=1}^n h_{i_1}h_{i_2}h_{i_3}
+ 6\sum_{i_1=1}^n\sum_{i_2=1}^n  h_{i_1}^2h_{i_2} +8\sum_{i=1}^n h_i^3 +
12n^2\sum_{i=1}^n \sum_{1 \leq j < k \leq n} h_ib_{jk}^2 +\\
&&\displaystyle + 24n^2\sum_{1 \leq j < k \leq n} (h_k+h_j)b_{jk}^2 +
48n^3 \sum_{1 \leq i<j < k \leq n} b_{ij}b_{ik}b_{jk} \\
\end{array}\]
Having in mind that $\sum_{i=1}^n h_i =0$ we obtain
\begin{lemma}\label{p111}
With the notation from \ref{aux-41} we have
\[ \begin{array}{rcl}
n^4(n+2)(n+4) p_{1,1,1,\Lambda} &=& \displaystyle
n^4(n+2)(n+4) \int_{S^{n-1}}
\left(f^\harm_{\Lambda,1}\right)^3 d\bar \mu\\
&=&  \displaystyle \sum_{i=1}^n h_i^3 + 3n^2\sum_{1
\leq i < j \leq n} (h_i+h_j)b_{ij}^2 + 6n^3 \sum_{1 \leq i<j < k \leq n}
b_{ij}b_{ik}b_{jk} \, .
\end{array}
\]
\end{lemma}
\subsection{Definition of $\Theta_{1,1,1,\Lambda}$}\label{def-t111}
Again we rescale our definition and find by Theorem \ref{thetam1mk}
for any integral lattice  $\Lambda$ the modular form
\[ \Theta_{1,1,1,\Lambda} = \sum_{i=1}^n
\Theta_{h_i, \Lambda}^3 + 3n^2 \sum_{1
\leq i < j \leq n} (\Theta_{h_i, \Lambda} + 
\Theta_{h_j, \Lambda})\Theta_{x_ix_j,\Lambda}^2 +  6n^3 \sum_{1 \leq i<j
< k \leq n} \Theta_{x_ix_j,\Lambda} \Theta_{x_ix_k,\Lambda}
\Theta_{x_jx_k,\Lambda}
\]
where $h_i$ denotes the harmonic polynomial $h_i=nx^2_i-\sum_{j=1}^mx_j^2$.
To determine the $q$-expansion of $\Theta_{1,1,1,\Lambda}$ we need
\begin{lemma}\label{aux-43}
Let $u,v,w \in \edop^n$ be three vectors in euclidean space. We fix with
$\Bcal_2 = \left\{ \frac{x^I}{\sqrt{I!}} \right\}_{I \subset \ndop^n \,
|I|=2}$ a basis of the homogeneous polynomials of degree 2
on $\edop^n$. Denote by $P=P_{\harm,2}$ be the harmonic projection of
degree two. Then we have an equality
\[ \Xi(u,v,w):=n^3(n+2)(n+4) \sum_{g_1,g_2,g_3 \in \Bcal_2}
P(g_1(u))P(g_2(v))P(g_3(w)) \int_{S^{n-1}} g_1g_2g_3 d\bar \mu = \qquad\]
\[ =
2\|u\|^2\|v\|^2\|w\|^2
-n(\|u\|^2\sca{v}{w}^2+\|v\|^2\sca{u}{w}^2+\|w\|^2\sca{u}{v}^2)
+n^2\sca{v}{w}\sca{u}{w}\sca{u}{v} \, . \]
Furthermore, when $\|u\|^2$, $\|v\|^2$, and $\|w\|^2$ are integers, and
the scalar products $\sca{v}{w}$, $\sca{u}{w}$, and $\sca{u}{v}$ are in
$\frac{1}{2}\zdop$, then $8 \Xi(u,v,w)$ is an integer.
\end{lemma}
\begin{proof}
This is a straightforward calculation along the lines of computing
$p_{1,1,1,\Lambda}$ in \ref{aux-41}.
\end{proof}

\begin{theorem}\label{theta-111}
For any integral lattice $\Lambda$ of level $N$ and discriminant $D$,
the modular form $\Theta_{1,1,1,\Lambda}$ defined in \ref{def-t111} has
level $N$, character $\left( \frac{D}{\cdot} \right)$, and
weight $\frac{n+12}{2}$. $\Theta_{1,1,1,\Lambda}$ is independent from
the embedding $\Lambda \to \edop^n$. Its $q$-expansion is given by
\[ \Theta_{1,1,1,\Lambda}(\tau) = n\sum_{k \geq 0} \left( \sum_{(u,v,w) \in
\Lambda^3 \, \|u\|^2+\|v\|^2+\|w\|^2=k} \Xi(u,v,w) \right) q^k \, . \]
Furthermore, $\frac{8}{n} \Theta_{1,1,1,\Lambda}(\tau) \in \zdop[[q]]$.
\end{theorem}
\begin{proof}
If we define $\Theta_{1,1,1,\Lambda}$ by
\[ \Theta_{1,1,1,\Lambda} = n^4(n+2)(n+4) \sum_{h1,h2,h_3 \in \Bcal^\harm_2 }
\Theta_{h_1,\Lambda} \Theta_{h_2,\Lambda}\Theta_{h_3,\Lambda} 
\int_{S^{n-1}}h_1h_2h_3 d \bar \mu \, ,
\]
then we obtain by Theorem \ref{thetam1mk} an invariant modular form of
the given weight, character, and level. Lemma \ref{p111} show that this
definition coincides with the definition in \ref{def-t111}.
By Proposition \ref{alt-def} we can use the basis $\Bcal_2 = \left\{
\frac{x^I}{\sqrt{I!}} \right\}_{I \subset \ndop^n \,
|I|=2}$ of the homogeneous polynomials of degree 2 and the harmonic
projection $P$ to compute $\Theta_{1,1,1,\Lambda}$ as
\[ \Theta_{1,1,1,\Lambda} = n^4(n+2)(n+4) \sum_{h1,h2,h_3 \in
\Bcal_2 }
\Theta_{P(g_1),\Lambda} \Theta_{P(g_2),\Lambda}\Theta_{P(g_3),\Lambda} 
\int_{S^{n-1}}g_1g_2g_3 d \bar \mu \, .
\]
Now we deduce the $q$-expansion and $\frac{8}{n} \Theta_{1,1,1,\Lambda}
\in \zdop[[q]]$ from Lemma \ref{aux-43}.
\end{proof}

\section{Some combinatorics}
\begin{lemma}\label{combi-0}
We define for all integers $d \geq 1$, and $w$ the quantity
\[ q_{d,w}=\sum_{k=0}^d (-1)^k \binom{d}{k} \binom{w+k}{d-1} \,.\]
For all integers $d \geq 1$, we have
$q_{d,-1}=1$, and $q_{d,w}=0$ for $w \ne -1$.
\end{lemma}
\begin{proof}
We consider the formal Laurent series $f_d,g_{d-1,w} \in \qdop((t))$,
given by
\[f_d:=\sum_{k=0}^d (-t\inv)^k\binom{d}{k} = \left( \frac{t-1}{t}
\right)^d \quad \mbox{, and } \quad
g_{a,b} := \sum_{k \in \zdop} \binom{b+k}{a} t^k \, . \]
We have $g_{0,b} = t^{-b}(1-t)\inv$ by definition.
From the formula $\binom{b+k}{a} = \binom{b-1+k}{a} + \binom{b-1+k}{a-1}$
we deduce that
$g_{a,b} = t(g_{a,b}+g_{a-1,b})$. Hence $g_{a,b}=\frac{t}{1-t}g_{a-1,b}$
which gives by induction:
\[ g_{a,b} = t^{a-b}(1-t)^{-1-a} \, .\]
The number $q_{d,w} $ is the coefficient of $t^0$ in $f_d g_{d-1,w}$. 
Eventually, we conclude from the above calculation that
$f_d g_{d-1,w} = t^{1+w}$.
\end{proof}

\begin{lemma}\label{pre-pm}
For all integers $w$, and $r \geq 0$ we have an equality
\begin{equation}\label{form-2}
\sum_{p=0}^r (-1)^{r-p}
(w+2p-2r)
\binom{w}{r-p} \binom{w+p-2r-1}{p}
= \left\{ \begin{array}{ll} w & \mbox{ for } r=0\\
0 & \mbox{ for } r \geq 1\\ \end{array} \right.
\end{equation}
\end{lemma}
\begin{proof}
First we denote the left hand side of equation (\ref{form-2}) by
$\xi_{r,w}$.
We consider the formal power series $f \in\zdop[[t]]$ given by:
\[ f = \sum_{p \geq 0} (w-2p)\binom{w}{p} (-t)^p=  w \sum_{p \geq 0}
\binom{w}{p} (-t)^p +2 \sum_{p \geq 0} \binom{w}{p}(-p) (-t)^p  \, . \]
Using the binomial equation $(1-t)^w=\sum_{p \geq 0} \binom{w}{p}
(-t)^p$, and its derivative with respect to $t$ we obtain
$\sum_{p \geq 0} \binom{w}{p}(-p) (-t)^p
= (-t) \frac{\dee}{\dee t} (1-t)^w =wt(1-t)^{w-1}$.
Therefore we have
\[  f = w(1-t)^w +2wt(1-t)^{w-1} = w(1+t)(1-t)^{w-1} \, .\]
Next we consider the formal power series $h \in \zdop[[t]]$ which we
define to be
\[ h= \sum_{p \geq 0} \binom{ (w-2r-1)+p}{p} t^p = g_{w-2r-1,w-2r-1} =
(1-t)^{2r-w} \,,  \]
where $g_{w-2r-1,w-2r-1}$ is the function defined in the proof of Lemma
\ref{combi-0}. Now we are able to compute $\xi_{r,w}$. Indeed, the
number $\xi_{r,w}$ is the coefficient of $t^r$ in $f \cdot h$.
Since $f \cdot h = w(1+t)(1-t)^{2r-1}$, we deduce from the binomial
formula that for $r \geq 1$ we have
\[\xi_{r,w} = w \left( (-1)^r\binom{2r-1}{r}
+(-1)^{r-1}\binom{2r-1}{r-1}   \right) = 0 \, .
\]
Since for $r=0$ the assertion also holds, we are done.
\end{proof}

\begin{lemma}\label{combi-1}
The rational numbers $r_{k,m}=\frac{1}{2^k k! \prod_{l=0}^{k-1}
(n+2m-4-2l)}$
fulfill the equation
\[ \sum_{k=0}^{d} \left(
\prod_{l=k}^{d-1} \frac{1}{(2l-2d)(n-2+2m-2d-2l)} \right) r_{k,m} =0 \,
\mbox{ for all } d \geq 1 \,. \]
\end{lemma}
\begin{proof}
We obtain from Lemma \ref{combi-0} that $q_{d,w+d-2} = 0$ for all
integers $w \geq 2-d$. 
We may write
\[ q_{d,w+d-2} = d \sum_{k=0}^d \frac{(-1)^k}{(d-k)!k!}
\prod_{l=k}^{k+d-2}(w+l)\, . \]
Considered as a polynomial in $w$ of degree at most $d-1$,
$q_{d,w+d-2}$ can have at most $d-1$ zeros. Hence it is identically zero
for all $w \in \qdop$. 
Setting $w=2-m-\frac{n}{2}$ we obtain therefore that
\[ \begin{array}{rcl}
0 &=& \frac{(-2)^{d-1}}{(-2)^d \prod_{l=0}^{2d-2} (n+2m-4-2l)}
\sum\limits_{k=0}^{d} \frac{(-1)^k}{(d-k)! k!}
\prod_{l=k}^{k+d-2} (w+l) \\ 
&=& \frac{1}{(-2)^d \prod_{l=0}^{2d-2} (n+2m-4-2l)}
\sum\limits_{k=0}^{d} \frac{(-1)^k}{(d-k)! k!}
\prod_{l=k}^{k+d-2} (n+2m-4-2l)     \\
&=& \frac{1}{(-2)^d} \sum\limits_{k=0}^{d} \frac{(-1)^k}{(d-k)! k!}
\frac{1}{\prod_{l=k+d-1}^{2d-2} (n+2m-4-2l)} 
\frac{1}{\prod_{l=0}^{k-1} (n+2m-4-2l)}     \\
&=& \frac{1}{(-2)^d} \sum\limits_{k=0}^{d} \frac{(-1)^k}{(d-k)! k!}
\frac{1}{\prod_{l=k}^{d-1} (n-2+2m-2d-2l)} 
\frac{1}{\prod_{l=0}^{k-1} (n+2m-4-2l)}     \\
& =& \sum_{k=0}^{d} \left(
\prod_{l=k}^{d-1} \frac{1}{(2l-2d)(n-2+2m-2d-2l)} \right) \left(
\frac{1}{2^k k! \prod_{l=0}^{k-1} (n+2m-4-2l)} \right)
\end{array}\]
This is the stated assertion.
\end{proof}

\begin{lemma}\label{pm}
The polynomials $p_m \in \qdop[c^2]$ which are implicitly defined by 
\[ \sum_{k=0}^m \frac{p_{m-k}}{a_{k,m}} = \frac{c^{2m}}{(2m)!}
\quad \mbox{ with the integers } \quad a_{k,m}=2^k k! 
\prod_{l=1}^{k} (n+4m-2k-2l) \]
from \ref{harm-num} (6) are explicitly given by
\begin{equation}\label{form-1}
 p_{m}(c) = \sum_{k=0}^m \frac{(-1)^k c^{2m-2k}}{(2m-2k)!k!2^k
\prod_{l=0}^{k-1} (n+4m-4-2l)}  \, . 
\end{equation}
\end{lemma}
\begin{proof}
We take the formula (\ref{form-1}) as the definition of $p_m$ and compute
the sum
$s_m:=\sum_{k=0}^m \frac{p_{m-k}}{a_{k,m}}$.
By definition $s_m$ is a polynomial in $\qdop[c^2]$ of degree at most
$2m$. We write $s_m = \sum_{r=0}^m \frac{t_r}{4^r (2m-2r)!} c^{2m-2r}$.
We find that
\[ t_r  = \sum_{p=0}^r \frac{(-1)^p}{
p!(r-p)!
\left( \prod_{l=1}^{r-p} (w+p-r-l) \right)
\left( \prod_{q=0}^{p-1} (w+2p-2r-2-q) \right) }  \, ,
\]
with $w:=\frac{n}{2}+2m-1$.
Up to the factor $(w+2p-2r)$ the two products in the denominator give
$\prod_{l=0}^{r} (w+p-r-l)$. So we get
\[ t_r =  \sum_{p=0}^r \frac{(-1)^p (w+2p-2r)}{
p!(r-p)!
\prod_{l=0}^{r} (w+p-r-l) } \]
Multiplying both side with the factor $\prod_{q=0}^{2r}(w-q) \ne 0$ we
get
\[ \left( \prod_{q=0}^{2r}(w-q) \right)  t_r =
 \sum_{p=0}^r \frac{(-1)^p (w+2p-2r) \prod_{q=0}^{r-p-1} (w-q)
\prod_{q=0}^{p-1}(w+p-2r-1-q)
}{
p!(r-p)!}\]
As usual, we define for a non-negative integer $k$ for all complex
numbers $z$ the binomial coefficient
$\binom{z}{k} =\frac{ \prod_{a=0}^{k-1} (z-a)}{k!}$.
Using this notation we obtain
\[ \left( \prod_{q=0}^{2r}(w-q) \right)  t_r =\sum_{p=0}^r 
(-1)^p \binom{w}{r-p} \binom{w+p-2r-1}{p} (w+2p-2r) \, .\]
Both side of this equation are polynomials of degree at most $2r$. By
Lemma \ref{pre-pm} the right hand side is zero for $r \geq 1$ and all
integers $w$. So it is zero for all $w$. We conclude that $t_r =0$ for
all $r \geq 1$. 
We finish the proof by checking $t_0=1$ which is obvious.
\end{proof}

\end{document}